\documentclass[a4paper,12pt,reqno]{amsart}
\usepackage[utf8]{inputenc}
\usepackage{amssymb}
\usepackage{amsmath}
\usepackage{amsthm}
\usepackage{enumerate}
\usepackage{hyperref}
\hypersetup{
	colorlinks=true,
	linkcolor=black,
	linkbordercolor={1 1 1},
	citecolor=black
}

\makeatletter
\@namedef{subjclassname@2020}{\textup{2020} Mathematics Subject Classification}
\makeatother

\newtheorem{thm}{Theorem}
\newtheorem{lem}[thm]{Lemma}

\theoremstyle{definition}
\newtheorem{remark}{Remark}

\newcommand{\IN}{\mathbb{N}}
\newcommand{\IP}{\mathbb{P}}

\newcommand{\IR}{\mathbb{R}}

\newcommand{\cH}{\mathcal{H}}
\newcommand{\cK}{\mathcal{K}}

\newcommand{\vol}{\mathrm{vol}}
\newcommand{\dist}{\mathrm{dist}}
\newcommand{\proj}{\mathrm{proj}}

\DeclareMathOperator{\conv}{conv}
\DeclareMathOperator{\len}{len}
\newcommand{\curv}{\mathcal{C}}
\newcommand{\aff}{\mathrm{aff}}
\newcommand{\dd}{\mathrm{d}}

\title[A limit theorem for random inscribed polytopes]{A limit theorem for Hausdorff approximation by random inscribed polytopes}

\author[M.~Sonnleitner]{Mathias Sonnleitner}
\address[]{Institute for Financial Mathematics and Applied Number Theory, 
Johannes Kepler University Linz,\\ Altenbergerstrasse 69, 4040 Linz, Austria}
\email{mathias.sonnleitner@jku.at}


\begin{document}

\begin{abstract}	
Approximate a smooth convex body $K$ with nonvanishing curvature by the convex hull of $n$ independent random points sampled from its boundary $\partial K$. In case the points are distributed according to the optimal density, we prove that the rescaled approximation error in Hausdorff distance tends to a Gumbel distributed random variable.  The proof is based on an asymptotic relation to covering properties of random geodesic balls on $\partial K$ and on a limit theorem due to Janson.

\end{abstract}

\maketitle

\section{Introduction and main result}

Let $K_n=\conv\{X_1,\dots,X_n\}$ be the convex hull of independent random points $X_1,\dots,X_n$ distributed on the boundary $\partial K$ of a smooth convex body $K\subset \IR^d$, i.e. a convex compact set with nonempty interior. The Hausdorff distance between $K_n$ and $K$ is given by
\begin{align*}
\delta_H(K_n,K)
=\max_{x\in \partial K}\min_{y\in K_n}\|x-y\|,
\end{align*}
where $\|\cdot\|$ denotes the Euclidean norm. Geometrically, it measures the maximal height of a cap of $K$ cut off by a facet of $K_n$. We are interested in the asymptotic behavior of the random variable $\delta_H(K_n,K)$ as $n\to \infty$. 

Assume that $K$ belongs to $\cK^3_+$, the set of convex bodies having three times continuously differentiable boundary with everywhere positive Gaussian curvature $\kappa_K$. If the random points $X_1,\dots,X_n$ are distributed on $\partial K$ according to a continuously differentiable density $h\colon \partial K\to (0,\infty)$ with respect to the $(d-1)$-dimensional Hausdorff measure $\cH^{d-1}$, then by \cite{GS96} (see \cite{Sch88} for the planar case) it holds that
\begin{equation}
	\label{eq:gs-smooth}
\hskip -5mm
\left(\frac{n}{\log n}\right)^{\frac{2}{d-1}}\delta_{H}(K_n,K)
\xrightarrow[]{\mathbb{P}}
\frac{1}{2}\left(
\frac{1}{\kappa_{d-1}}\max_{x\in\partial K}\frac{\sqrt{\kappa_{K}(x)}}{h(x)}
\right)^{\frac{2}{d-1}},
\end{equation}
where $\kappa_{d-1}$ is the volume of the $(d-1)$-dimensional Euclidean unit ball and $\overset{\mathbb{P}}{\to}$ denotes convergence in probability. Moreover, the right-hand side of \eqref{eq:gs-smooth} is minimal for the density 
\begin{equation} \label{eq:h-opt}
h_{\kappa}(x)
=\frac{\sqrt{\kappa_K(x)}}{v_{\kappa}(K)}, \quad x\in \partial K, \qquad \text{where }
v_{\kappa}(K):=\int_{\partial K} \sqrt{\kappa_K(x)}\dd \cH^{d-1}(x).
\end{equation}

\begin{remark}
The convergence \eqref{eq:gs-smooth} holds also in expectation, see \cite{DW96,PSS+24}. 
\end{remark}

The density $h_{\kappa}$  also appears when studying optimal point sets. Namely, the sequence of optimal point sets $P_n^*=\{x_1,\dots,x_n\}$ minimizing $\delta_H(\conv(P_n),K)$ over all $n$-point sets $P_n\subset \partial K$ is uniformly distributed with density $h_{\kappa}$, that is, 
\[
\lim_{n\to \infty}\frac{|P_n^*\cap A|}{n} =\int_A h_{\kappa}(x)\dd\cH^{d-1}(x)\qquad \text{ for any Borel set }A\subset \partial K,
\]
and, similar to \eqref{eq:gs-smooth} for $h=h_{\kappa}$, the optimal points $P_n^*$ satisfy 
\[
\lim_{n\to \infty}n^{\frac{2}{d-1}}\delta_H(\conv(P_n^*),K)
=\frac{1}{2}\left( \frac{\vartheta_{d-1}}{\kappa_{d-1}}v_{\kappa}(K)
\right)^{\frac{2}{d-1}},
\]
where $\vartheta_{d-1}$ is the minimal covering density of $\IR^{d-1}$ by unit balls, see \cite{GS96,Gru93,Sch81}. 

If the error of the random approximation of $K$ by $K_n$ is measured in terms of missed volume instead of Hausdorff distance, then more precise asymptotics are available. For example, there is a central limit theorem due to \cite{Tha18}, see also the variance bounds in \cite{Rei03,RVW08} and related central limit theorems in \cite{BR10,BT20,CSY13,Hsi94,LSY19,Rei05,TTW18,Vu06}. Moreover, to minimize the volume difference, it is optimal to sample proportional to $\kappa_K^{1/(d+1)}$. Random points sampled according to this density behave similarly to optimal points, particularly in high dimension, see \cite{PSW22,Rei02,SW03}.

If $K_n$ is the convex hull of independent random points distributed in the interior of $K$, it was proven in \cite{BHB98} for $d=2$ that, for suitable scalars $a_n$ and $b_n$, as $n\to \infty$,
\begin{equation} \label{eq:extreme-limit}
	\frac{\delta_{H}(K_n,K)-a_n}{b_n}\xrightarrow[]{\rm d} G,
\end{equation}
where $G$ is a standard Gumbel random variable, i.e. $\IP[G\le x]=\exp(-e^{-x})$, $x\in \IR$, and $\xrightarrow[]{\rm d}$ denotes convergence in distribution. Thus, the random variables $\delta_H(K_n,K), n\in \IN,$ exhibit Gumbel fluctuations similarly to maxima of independent random variables. The proof in \cite{BHB98} relies on a blocking technique together with results from extreme-value theory for weakly dependent random variables. For the case $d\ge 3$ we refer to \cite{CY25+}.

Our main result states that a limit theorem of the form \eqref{eq:extreme-limit} holds provided the random points are distributed on $\partial K$ with respect to the density $h_{\kappa}$ as in \eqref{eq:h-opt}. 

\begin{thm}\label{thm:main}
Let $K\in \cK^3_+$. For each $n\in \IN$, let $X_1,\dots,X_n$ be independently distributed on $\partial K$ with density $h_{\kappa}$ as given in \eqref{eq:h-opt} and let $K_n=\conv\{X_1,\dots,X_n\}$.  Let 
\[
c_n=\Big(\frac{v_{\kappa}(K)}{\kappa_{d-1}}\frac{\log n}{n}\Big)^{\frac{2}{d-1}} \qquad \text{and}\qquad \alpha=\frac{1}{(d-1)!}\Big(\frac{\sqrt{\pi}\Gamma(\frac{d+1}{2})}{\Gamma(\frac{d}{2})}\Big)^{d-2},
\]
where $v_{\kappa}(K)=\int_{\partial K} \sqrt{\kappa_K(x)}\dd \cH^{d-1}(x)$ is as in \eqref{eq:h-opt}. Then \eqref{eq:extreme-limit} holds with
\[
a_n=\Big(\frac{1}{2}+\frac{(d-2)\log\log n+\log \alpha}{(d-1)\log n}\Big)\cdot c_n \qquad \text{and}\qquad b_n=\frac{c_n}{(d-1)\log n}.
\]
\end{thm}

In particular, Theorem~\ref{thm:main} implies the convergence in \eqref{eq:gs-smooth} for $h=h_{\kappa}$. It remains open to extend it to other sampling densities.

The proof of Theorem~\ref{thm:main} is based on a limit theorem of the form \eqref{eq:extreme-limit} for maximal spacings due to \cite{Jan87}, see Lemma~\ref{lem:janson} below, and an asymptotic relation between the geodesic covering radius of a point set on the boundary $\partial K$ and the Hausdorff distance between its convex hull to $K\in \cK^3_+$. Here, for any finite point set $P\subset \partial K$, the geodesic covering radius $\varrho(P)$ is given by
\[
\varrho(P)
=\inf\Big\{\varrho>0\colon \partial K\subset \bigcup_{x\in P} B_{\gamma}(x,\varrho)\Big\},
\]
where $B_{\gamma}(x,\varrho)$ denotes a geodesic ball on $\partial K$ with center $x\in \partial K$ and radius $\varrho>0$ with respect to the metric $\gamma$ induced by the second fundamental form, as explained below.  The proof of \eqref{eq:gs-smooth} relies on the asymptotic relation
\begin{equation*} 
\delta_H(\conv(P),K)
=\frac{\varrho(P)^2}{2}(1+o(1))\qquad \text{as}\qquad \varrho(P)\to 0,
\end{equation*}
proven in \cite[Theorem~4]{Gru93}. For the proof of Theorem~\ref{thm:main} we require the following refinement.

\begin{thm}\label{thm:hd-covering}
Let $K\in \cK^3_+$. There exist $\varrho>0$ and $C>0$ such that, for any finite point set $P\subset \partial K$ with $\varrho(P)\le \varrho$,
\[
\Big|\delta_H(\conv(P),K)-\frac{\varrho(P)^2}{2}\Big|\le C\varrho(P)^3.
\]
\end{thm}
	If $K$ is a Euclidean ball and $\varrho(P)$ is taken with respect to the Euclidean metric, then Theorem~\ref{thm:hd-covering} holds even with $C=0$, see e.g.~\cite[Lemma 7.1]{BW03}, and Theorem~\ref{thm:main} can be deduced from the limit theorem in \cite[Theorem 3.2]{Ser23} using the Delta Method. 

If $K$ is a polytope, only the distance from $K_n$ to the vertices of $K$ contributes to the Hausdorff distance. In case of $d=2$, a limit distribution containing one factor per vertex was obtained in \cite{BHB98} for points distributed inside and in \cite{PSS+24} for points distributed on the boundary.  It remains open to obtain analogous results in higher dimensions.

\begin{remark}
The related problem of deriving a limit theorem for the diameter of the random convex hull $K_n$ was solved in \cite{MM07} for points sampled uniformly in the unit ball or on the unit sphere. The limiting distribution is Weibull, and Gumbel if the dimension grows with the number of points \cite{HK25}. 
\end{remark}

\section{Proofs}

In the following, as in \cite{GS96,Gru93}, we equip $\partial K$ with a Riemannian structure induced by the second fundamental form. This is combined with a uniform local approximation of the boundary $\partial K$ by osculating paraboloids, as used for example in \cite{Rei05}, to get uniform bounds on the metric structure of $\partial K$. We deduce Theorem~\ref{thm:hd-covering} from a local version for individual caps. We prove volume asymptotics for small geodesic balls in Lemma~\ref{lem:vol-geo} below, assuming only $K\in \cK^3_+$, and then use the Delta Method to deduce Theorem~\ref{thm:main} from the limit theorem for maximal spacings in \cite{Jan87}.

Let $K\in \cK^3_+$ and $p\in \partial K$. Then there is a unique affine hyperplane $H_K(p)$ supporting $\partial K$ at $p$. Write $\proj_p$ for the orthogonal projection onto $H_K(p)$. Let $U=U^{(p)}$ be an open neighborhood of $p$ in $\partial K$ such that $\proj_p$ maps $U$ homeomorphically onto an open ball $U'=\proj_p(U)\subset H_K(p)$.  Identifying $H_K(p)$ with $\IR^{d-1}$ with $p$ as the origin, parametrize $U$ in the form $U=\{(u,f^{(p)}(u))\colon u\in U'\}$, where $f^{(p)}(u)$ is the distance of $u\in U'$ to its preimage $\proj_p^{-1}(u)\in U$. By assumption, the function $f^{(p)}$ is three times continuously differentiable and convex on $U'$. Moreover, see \cite[Section~5]{Rei05}, for any $\delta>0$ there exists $\lambda>0$ such that, for all $p\in \partial K$ and all $y\in B^{(p)}(0,4\lambda):=\{x\in H_K(p)\colon \|x-p\|\le 4\lambda\}$,
\begin{equation} \label{eq:osculating}
(1+\delta)^{-1}b_0^{(p)}(y)
\le f^{(p)}(y)
\le (1+\delta)b_0^{(p)}(y),
\end{equation}
where $b_0^{(p)}$ is the quadratic form induced by the Hessian of $f^{(p)}$ at $0$ and given by
\[
b_0^{(p)}(y)=\frac{1}{2}\sum_{i,j}f_{ij}^{(p)}(0)y_iy_j,\qquad y\in \IR^{d-1}.
\]
Here and in the following, $f^{(p)}_{ij}$ and $f^{(p)}_k$ denote partial derivatives of $f^{(p)}$ and indices $i,j,k$ run from one to $d-1$. The Gaussian curvature $\kappa_K(p)$ at $p\in \partial K$ is given by the determinant of the matrix $(f_{ij}^{(p)}(0))_{i,j}$ and remains uniformly bounded between two positive constants. Thus, for any $p\in \partial K$,
\begin{equation} \label{eq:paraboloid-lower}
c_K \|y\|^2
\le b_0^{(p)}(y)
\le C_K \|y\|^2,\qquad y\in \IR^{d-1},
\end{equation}
where $c_K,C_K>0$ depend only on $K$. Let $\delta=1$ and fix $\lambda$ according to \eqref{eq:osculating}. For each $p\in \partial K$, choose $U^{(p)}$ such that $\proj_{p}(U^{(p)})=B^{(p)}(0,\lambda)$. 

For $p\in \partial K$ and $u\in B^{(p)}(0,2\lambda)$ the second fundamental form is given by 
\[
q_{u}(s)
=q_{u}^{(p)}(s)
=\sum_{i,j}\frac{f_{ij}(u)}{\sqrt{1+\sum_{k}f_k(u)^2}}s_i s_j\qquad s\in \IR^{d-1}.
\]
Since $\kappa_K>0$, the quadratic form $q_{u}^{(p)}$ is positive definite. This holds also in the following uniform sense. 

\begin{lem}\label{lem:pd-uniform}
It holds for all $p\in \partial K$ and all $u\in B^{(p)}(0,2\lambda)$ that
\[
	q_{u}^{(p)}(s)
	\ge \frac{c_K}{2}\|s\|^2,\qquad s\in \IR^{d-1},
\]
where $c_K>0$ is as in \eqref{eq:paraboloid-lower}.
\end{lem}
\begin{proof}
Let $p\in \partial K$ and	identify $H_K(p)$ with $\IR^{d-1}$. Let $U=U^{(p)}$ and $f=f^{(p)}$ be as above. Let $u\in B(0,2\lambda)=B^{(p)}(0,2\lambda)$ and $x=(u,f(u))$.  For $h\in \IR^{d-1}$ with $u+h\in B(0,3\lambda)$, define
\[
y=\Big(u+h,f(u)+\sum_{k}f_k(u)h_k \Big)\in H_K(x).
\]
Let $x_1=(u+h,f(u+h))\in \partial K$. By Taylor's theorem we find $\xi\in [0,1]$ such that
\begin{equation} \label{eq:dist-xi}
\|x_1-y\|
=f(u+h)-f(u)-\sum_{k}f_k(u)h_k 
=\frac{1}{2}\sum_{i,j}f_{ij}(u+\xi h)h_ih_j.
\end{equation}

By \eqref{eq:osculating}, the angle $\beta\in (0,\frac{\pi}{2})$ enclosed between $H_K(x)$ and $H_K(p)=\IR^{d-1}$ satisfies 
\begin{equation} \label{eq:cos-lower}
\frac{1}{\sqrt{1+\sum_{k}f_k(u)^2}}
=\cos\beta
\ge c_0
\end{equation}
for some absolute constant $c_0>0$. Thus,
\begin{equation} \label{eq:y-x}
\|x-y\|
\le \frac{\|h\|}{\cos \beta}
\le c_0^{-1}\|h\|
\end{equation}
and it follows from \eqref{eq:dist-xi} that $y_1=\proj_x(x_1)\in H_K(x)$ satisfies
\begin{align*}
\|y_1-y\|
\le \|x_1-y\|
\le (d-1)\max_{i,j}\sup_{z\in B(0,3\lambda)} |f_{ij}(z)|\cdot \|h\|^2.
\end{align*}
Using \eqref{eq:y-x} and continuity of the second derivatives of $f$ on $B(0,3\lambda)$, we find $\delta>0$ such that
\[
\|x-y_1\|
\le \|x-y\|+\|y-y_1\|
< \lambda,\qquad \text{whenever }\|h\|<\delta. 
\]
Thus, for $h\in \IR^{d-1}$ with $\|h\|\le \delta$ it holds that $x_1\in U^{(x)}$, which is parametrized via $f^{(x)}$ and $H_K(x)$. Moreover, the previous assumption $u+h\in B(0,3\lambda)$ is satisfied. By uniform continuity of the second derivatives of $f$ on $B(0,3\lambda)$ we can additionally assume that $\delta$ is small enough such that 
\begin{equation} \label{eq:f2-uniform}
f_{ij}(u+\xi h)\le 2\,f_{ij}(u),\qquad \|h\|\le \delta.
\end{equation}

Since $y_1$ is the orthogonal projection of $x_1$ onto $H_K(x)$, combining \eqref{eq:paraboloid-lower}, \eqref{eq:dist-xi} and \eqref{eq:f2-uniform}, we deduce that, for any $h\in \IR^{d-1}$ with $\|h\|=\delta$,
\[
q^{(p)}_u(h)
\ge \frac{1}{2}\sum_{i,j}\frac{f_{ij}(u+\xi h)}{\sqrt{1+\sum_k f_k(u)^2}}h_i h_j
= \|x_1-y_1\|
\ge \frac{c_K}{2} \|x-y_1\|^2 
\ge \frac{c_K}{2} \|h\|^2,
\]
where $c_K>0$ is as in \eqref{eq:paraboloid-lower}. By homogenity, this completes the proof.
\end{proof}

The assumption $K\in \cK^3_+$ implies that the Riemannian metric induced by the second fundamental form is of class $C^1$ on $\partial K$. In this metric, the length of a continuously differentiable curve $k\colon [0,1]\to U$ connecting two points $k(0)=x$ and $k(1)=y$ in $U=U^{(p)}$ for some $p\in \partial K$ is given by 
\[
\ell(k)
:= \int_0^1 \sqrt{q_{u(\tau)}(u'(\tau))} \dd \tau,
\]
where $u=\proj_p(k)$.  If $x$ and $y$ do not both belong to $U$, we dissect the curve suitably and add. The geodesic distance $\gamma(x,y)$ between $x,y\in \partial K$ is given by the infimum of lengths over all continuously differentiable curves in $\partial K$ connecting $x$ and $y$. 

If $A\subset U$ is measurable with $A'=\proj_p(A)$, then its Riemannian volume is given by
\[
\vol_{\gamma}(A)
=\int_{A'}\sqrt{\det q_{u}}\dd u
=\int_{A}\sqrt{\kappa_K(x)}\dd \cH^{d-1}(x),
\]
and, in general, for $A\subset \partial K$ we dissect and add. Thus, the Riemannian volume $\vol_{\gamma}$ has density $h_{\kappa}$ as in \eqref{eq:h-opt} with respect to $\cH^{d-1}$ on $\partial K$. Let
\[
B_{\gamma}(x,r)
=\{y\in \partial K\colon \gamma(x,y)\le r\}
\]
be the (closed) geodesic ball with center $x\in \partial K$ and radius $r>0$. 

\begin{lem}\label{lem:small-ball}
	For every $p\in \partial K$, it holds that $B_{\gamma}(p,\varrho_0)\subset U^{(p)}$ with
	$\varrho_0=\lambda\sqrt{c_K}/5$, where $c_{K}$ is as in \eqref{eq:paraboloid-lower}.
\end{lem}
\begin{proof}
Let $p\in \partial K$ and let $U=U^{(p)}$ be as above with $\proj_{p}(U)=B(0,\lambda)$, where we identify $H_K(p)$ with $\IR^{d-1}$. Let $x\in \partial K$ with $\gamma(x,p)\le \varrho_0$ and connect $p$ with $x$ by a continuously differentiable curve $k$ in $\partial K$ of length $\ell(k)\le 2\varrho_0$ with $k(0)=p$ and $k(1)=x$.  Let $\alpha\in [0,1]$ be such that $k([0,\alpha])\subset U$. Then the orthogonal projection $u=\proj_p(k)$ satisfies $u([0,\alpha])\subset B(0,\lambda)$. It holds by Lemma~\ref{lem:pd-uniform} that 
\begin{equation} \label{eq:curve-lower}
\ell(k)
\ge \int_0^{\alpha}\sqrt{q_{u(\tau)}(u'(\tau))}\dd\tau
\ge \sqrt{\frac{c_K}{2}}\int_0^{\alpha}\|u'(\tau)\|\dd\tau.
\end{equation}
The integral on the right-hand side of \eqref{eq:curve-lower} equals the arc length of the curve $u\colon [0,\alpha]\to B(0,\lambda)$.  In particular, 
\[
\ell(k)
\ge  \sqrt{\frac{c_K}{2}}\|p-u(\tau)\|,\qquad \tau\le \alpha.
\]
Then $\|p-u(\tau)\|\le \frac{2\sqrt{2}\varrho_0}{\sqrt{c_K}}< \frac{2}{3}\lambda$ for $\tau\le \alpha$ and the curve $u$ remains within $B(0,\frac{2}{3}\lambda)$. Since $u$ was the restriction of $\proj_p(k)$ to $B(0,\lambda)$, the curve $k$ remains within $U$ and thus $x\in U$. This proves that $B_{\gamma}(p,\varrho_0)\subset U$.
\end{proof}

Write $\dist(y,A)=\inf_{x\in A}\|x-y\|$ for the Euclidean distance of a point $y$ to a set $A$. We have the following local version of Theorem~\ref{thm:hd-covering}. 

\begin{lem}\label{lem:hd-covering}
There exist $\varrho>0$ and $C>0$ such that, for any $x,y\in \partial K$, if $\gamma(x,y)\le \varrho$, then the Euclidean distance between $y$ and the support hyperplane $H_K(x)$ of $\partial K$ at $x$ satisfies
\[
\Big|\dist\big(y,H_K(x)\big)-\frac{\gamma(x,y)^2}{2}\Big|\le C\gamma(x,y)^3.
\]
\end{lem}
\begin{proof}
Let $\varrho_0$ be as in Lemma~\ref{lem:small-ball}. Fix $p_1,\dots,p_m\in \partial K$ such that the geodesic balls $B_{\gamma}(p_{1},\varrho_0),\dots,B_{\gamma}(p_{m},\varrho_0)$ cover $\partial K$. This is possible by compactness of $\partial K$ and the fact that the topology induced by $\gamma$ coincides with the one inherited by the ambient space $\IR^d$. By Lemma~\ref{lem:small-ball}, the neighborhoods $U^{(p_1)},\dots,U^{(p_{m})}$ cover $\partial K$. Moreover, for any $\ell=1,\dots,m$, the derivatives of $f^{(p_{\ell})}$ up to third order are uniformly bounded on $B^{(p_{\ell})}(0,2\lambda)$ by a constant only depending on $K$ (and $p_1,\dots,p_m$).

	Let $x,y\in \partial  K$ with $\gamma(x,y)\le \varrho$, where $\varrho\le \frac{\varrho_0}{2}$ will be chosen later. In particular, both $x$ and $y$ belong to some $U=U^{(p_\ell)}$ and the orthogonal projections $x'=\proj_{p_{\ell}}(x)$ and $y'=\proj_{p_{\ell}}(y)$ onto $H_K(p_{\ell})=\IR^{d-1}$ are contained in $B(0,\lambda)=B^{(p_{\ell})}(0,\lambda)$. The function $f=f^{(p_{\ell})}$ represents $U$ in the form $(u,f(u))$ where $u\in B(0,\lambda)$.  Since 
\begin{equation} \label{eq:far-boundary}
	\gamma(x,y)\le \frac{\varrho_0}{2} \le \inf_{z\in \partial K\setminus U}\gamma(x,z),
\end{equation}
a curve connecting $x$ and $y$ and having minimal length must be contained in $U$. Let 
\[
\curv=\{u\in C^1([0,1];B(0,2\lambda))\colon u(0)=x',u(1)=y'\}
\]
be the set of continuously differentiable curves $u\colon [0,1]\to B(0,2\lambda)$ connecting $x'$ and $y'$. Moreover, for $z_1,z_2\in B(0,2\lambda)$, define a $d-1$ by $d-1$ matrix $G(z_1,z_2)$ with entries
\[
G_{ij}(z_1,z_2)
=\frac{f_{ij}(z_1)}{\sqrt{1+\sum_{k}f_k(z_2)^2}} 
\]
such that, by \eqref{eq:far-boundary}, the geodesic distance between $x$ and $y$ satisfies 
\begin{equation} \label{eq:geo}
\gamma(x,y)
=\inf_{u\in  \curv}L(u),
\qquad \text{} L(u):=\int_0^1 \sqrt{\langle G(u(\tau),u(\tau))\cdot u'(\tau),u'(\tau)\rangle} \dd \tau,
\end{equation}
where $\langle \cdot,\cdot\rangle$ denotes the standard Euclidean inner product.  By Lemma~\ref{lem:pd-uniform} and uniform boundedness of the first and second derivatives of $f$, it holds that  
\begin{equation} \label{eq:G-bounds}
c_G\|s\|^2\le \langle G(z_1,z_2)s,s\rangle \le C_G\|s\|^2,\qquad z_1,z_2\in B(0,2\lambda), s\in \IR^{d-1},
\end{equation}
for some $c_G,C_G>0$ depending only on $K$ (and $p_1,\dots,p_m$). Thus, for every $u\in \curv$,  
\[
\sqrt{c_G}\len(u) \le 
L(u)\le \sqrt{C_G}\len(u),\qquad \text{ where }
\len(u)
:= \int_0^1 \|u'(\tau)\|\dd \tau
\]
is the arc length of the curve $u$. In particular, this holds for the straight line segment connecting $x'$ and $y'$ which is parametrized by $\bar{u}\in \curv$ with
\begin{equation} \label{eq:bar-u}
\bar{u}(\tau):=x'+\tau(y'-x'),\qquad \tau\in [0,1],
\end{equation}
and has minimal arc length 
\begin{equation*} 
\len(\bar{u})
=\|x'-y'\|
=\inf_{u\in \curv}\len(u).
\end{equation*}

By standard methods in the calculus of variation, see e.g.~\cite[Chapter 7]{Dac25},  the infimum in \eqref{eq:geo} is attained by some $u^*\in \curv$.  The arc length of the curve $u^*$ satisfies
\begin{equation} \label{eq:arc-min}
	\sqrt{c_G}\|x'-y'\|
	\le \sqrt{c_G}\len(u^*)
\le \gamma(x,y)
\le L(\bar{u})
\le \sqrt{C_G}\|x'-y'\|
\end{equation}
and thus the curve $u^*$ is fully contained in the Euclidean ball
\begin{equation} \label{eq:ball-x-y}
	B:=B\Big(x',\sqrt{\frac{C_G}{c_G}}\|x'-y'\|\Big).
\end{equation}
Choose $\varrho\in (0,\frac{\varrho_0}{2})$ such that $\sqrt{\frac{C_G}{c_G}}\varrho<\lambda$. Then $B\subset B(0,2\lambda)$. 

The distance $\dist(y,H_K(x))$ of $y$ to the supporting hyperplane $H_K(x)$ of $\partial K$ at $x$ is given by Taylor's theorem as
\begin{align} \label{eq:dH-F}
\dist(y,H_K(x))
&=\frac{f(y')-f(x')-\sum_{k}f_k(x')(y'_k-x'_k)}{\sqrt{1+\sum_{k}f_k(x')^2}}\notag \\
&= \frac{1}{2}\sum_{i,j}\frac{f_{ij}(x'+\xi(y'-x'))}{\sqrt{1+\sum_{k}f_k(x')^2}}(y'_i-x'_i)(y'_j-x'_j),
\end{align}
for some $\xi\in [0,1]$. For $\xi$ as in \eqref{eq:dH-F}, define 
\begin{equation} \label{eq:geo-bar}
\bar{\gamma}(x,y)
=\inf_{u\in \curv} \bar{L}(u),
\qquad
\bar{L}(u):=\int_0^1 \sqrt{\langle G(x'+\xi(y'-x'),x')\cdot u'(\tau),u'(\tau)\rangle} \dd \tau.
\end{equation}
The infimum in the definition of $\bar{\gamma}(x,y)$ is attained by some $u\in C$ with constant derivative, that is, given by $\bar{u}$ as in \eqref{eq:bar-u}, see e.g.~\cite[Chapter 7]{Dac25}. This implies 
\begin{align}\label{eq:gamma-bar}
\bar{\gamma}(x,y)
=\bar{L}(\bar{u})
&=\sqrt{\langle G(x'+\xi(y'-x'),x')\cdot (y'-x'),y'-x'\rangle}\notag\\
&=\sqrt{2\,\dist(y,H_K(x))}.	
\end{align}
It follows from \eqref{eq:G-bounds} that
\begin{equation} \label{eq:gamma-bar-bound}
	\sqrt{c_G}\|x'-y'\|
	\le \bar{\gamma}(x,y) 
	\le \sqrt{C_G}\|x'-y'\|.
\end{equation}

It remains to estimate the difference between $\bar{\gamma}(x,y)$ and $\gamma(x,y)$. By optimality of $u^*$ for \eqref{eq:geo} and $\bar{u}$ for \eqref{eq:geo-bar}, it holds that
\[
\gamma(x,y)
=L(u^*)
\le L(\bar{u}) 
\le \bar{\gamma}(x,y) + |L(\bar{u})-\bar{L}(\bar{u})|,
\]
and analogously with $\gamma, L$ and $u^*$ exchanged with $\bar{\gamma},\bar{L}$ and $\bar{u}$. This gives
\begin{equation} \label{eq:L-bounds}
|\gamma(x,y)-\bar{\gamma}(x,y)|
\le \max\big\{|L(\bar{u})-\bar{L}(\bar{u})|,|L(u^*)-\bar{L}(u^*)|\big\}.
\end{equation}
Let $u\in \{u^*,\bar{u}\}$. Then $u$ is contained in the ball $B$ as in \eqref{eq:ball-x-y}. Note that, for any $z_1,\dots,z_4\in B(0,2\lambda)$ and $s\in \IR^{d-1}$,
\begin{align}\label{eq:sqrt-Gs}
	\big|\sqrt{\langle G(z_1,z_2)\cdot s,s\rangle} -\sqrt{\langle G(z_3,z_4)\cdot s,s\rangle}\big|
	&\le \frac{|\langle (G(z_1,z_2)-G(z_3,z_4))\cdot s,s\rangle|}{\sqrt{\langle G(z_1,z_2)\cdot s,s\rangle} +\sqrt{\langle G(z_1,z_2)\cdot s,s\rangle}}\notag\\
	&\le \frac{\|G(z_1,z_2)-G(z_3,z_4)\|_{2\to 2}}{2\sqrt{c_G}}\|s\|\notag\\
	&\le C\, \max_{i,j}|G_{ij}(z_1,z_2)-G_{ij}(z_3,z_4)|\,\|s\|,
\end{align}
where $\|\cdot\|_{2\to 2}$ denotes the spectral norm. Here and in the remainder of the proof, write $C$ for a positive constant depending only on $K$ whose value may change from line to line.  Using \eqref{eq:sqrt-Gs}, repeated application of Taylor's theorem and the uniform bounds on the derivatives of $f$ gives
\begin{align*}
|L(u)-\bar{L}(u)|
&\le C \max_{i,j}\sup_{z\in B}|G_{ij}(z,z)-G_{ij}(x'+\xi(y'-x'),x')|\,\len(u)\notag\\
&\le C\|x'-y'\|\,\len(u).
\end{align*}
Thus, it follows from \eqref{eq:arc-min} and \eqref{eq:L-bounds} that
\begin{equation}\label{eq:gammas}
	|\gamma(x,y)-\bar{\gamma}(x,y)|
	\le C \|x'-y'\|^2.
\end{equation}
Recall that $\bar{\gamma}(x,y)=\sqrt{2\,\dist(y,H_K(x))}$. Thus, by \eqref{eq:gamma-bar-bound} and \eqref{eq:gammas},
\begin{align*}
\Big|\frac{\gamma(x,y)^2}{2}-\dist(y,H_K(x))\Big|
&\le |\gamma(x,y)+\bar{\gamma}(x,y)||\gamma(x,y)-\bar{\gamma}(x,y)|\\
&\le C \gamma(x,y)^3.
\end{align*}
This completes the proof of Lemma~\ref{lem:hd-covering}.
\end{proof}

\begin{proof}[Proof of Theorem~\ref{thm:hd-covering}]
	Choose $\varrho$ as in Lemma~\ref{lem:hd-covering} and let $P$ be a finite point set on $\partial K$ with $\varrho(P)\le \varrho$. If the inscribed polytope $\conv(P)$ has empty interior, the statement is trivial. Thus, assume that $\conv(P)$ has nonempty interior.  For every facet $F$ of $\conv(P)$, the support hyperplane of $K$ with the same outward pointing normal vector as $F$ is given by $H_K(z_F)$ for some $z_F\in \partial K$. Let $h_F$ be the Euclidean distance between the parallel affine hyperplanes $H_K(z_F)$ and $\aff F$ (affine hull of $F$). The intersection of $\partial K$ with the closed half-space bounded by $\aff F$ which does not intersect the interior of $\conv(P)$ defines a cap
\[
C^{\partial K}(z_F,h_F)
=\{y\in \partial K\colon \dist(y,H_K(z_F))\le h_F\}
\]
of $\partial K$ with center $z_F$ and height $h_F$. Since $z\mapsto \min_{x\in P}\gamma(z,x)$ is continuous on $C^{\partial K}(z_F,h_F)$, there is a point $x_F\in C^{\partial K}(z_F,h_F)$ with 
\[
\varrho_F
:=\min_{x\in P}\gamma(x_F,x)
=\max_{z\in C^{\partial K}(z_F,h_F)}\min_{x\in P}\gamma(z,x).
\]
Since the caps $\{C^{\partial K}(z_F,h_F)\colon F \text{ is a facet of }\conv(P)\}$ cover $\partial K$, it holds that
\begin{equation} \label{eq:max-F}
\varrho(P)
=\max_F \varrho_F \qquad \text{and}\qquad 
\delta_H(\conv(P),K)
=\max_F h_F.
\end{equation}

Let $F$ be a facet of $\conv(P)$. The geodesic ball $B_{\gamma}(x_F,\varrho_F)\subset \partial K$ has maximal radius among all geodesic balls which are centered in $C^{\partial K}(z_F,h_F)$ and do not contain any point of $P$ in their interior. Thus, the ball $B_{\gamma}(z_F,\varrho_F)$ intersects $P$ and in particular $\aff F$. Let $y_1\in B_{\gamma}(z_F,\varrho_F)\cap \aff F$.
By \eqref{eq:max-F},
\[
\gamma(y_1,z_F)\le \varrho_F\le \varrho(P)\le \varrho
\]
and thus, by Lemma~\ref{lem:hd-covering}, 
\begin{equation} \label{eq:rho-h-1}
h_F
=\dist(y_1,H_K(z_F))
\le \frac{\gamma(y_1,z_F)^2}{2}+C\gamma(y_1,z_F)^3
\le \frac{\varrho_F^2}{2}+C\varrho_F^3.
\end{equation}
In order to bound $h_F$ from below in terms of $\varrho_F$, consider the cap
\[
C^{\partial K}(x_F,h_F)
=\{y\in \partial K\colon \dist(y,H_K(x_F))\le h_F\}
\]
which is cut off by a translated copy of the support hyperplane $H_K(x_F)$ of $\partial K$ at $x_F$.
The cap $C^{\partial K}(x_F,h_F)$ intersects $P$ because $C^{\partial K}(z_F,h_F)$ has maximal height among all caps which are centered in $C^{\partial K}(z_F,h_F)$ and do not contain any point of $P$ in their interior. Since $B_{\gamma}(x_F,\varrho_F)$ intersects $P$ only in its boundary, it cannot contain $C^{\partial K}(x_F,h_F)$ in its interior. Thus, we find $y_2\in C^{\partial K}(x_F,h_F)$ with $\gamma(y_2,x_F)=\varrho_F$. Again by Lemma~\ref{lem:hd-covering},
\begin{equation} \label{eq:rho-h-2}
\frac{\varrho_F^2}{2}-C\varrho_F^3
=\frac{\gamma(y_2,x_F)^2}{2}-C\gamma(y_2,x_F)^3
\le \dist(y_2,H_K(x_F))
\le h_F.
\end{equation}
Combining \eqref{eq:max-F}, \eqref{eq:rho-h-1} and \eqref{eq:rho-h-2} completes the proof of Theorem~\ref{thm:hd-covering}.
\end{proof}

We require the following volume asymptotics for small geodesic balls on $\partial K$.  In case of infinite smoothness, these are well-known and more precise asymptotics are available, see e.g.~\cite[Theorem~3.1]{Gra74} or \cite[Lem.~2.3]{CCN21}. 

\begin{lem}\label{lem:vol-geo}
Let $K\in \cK^3_+$. There are $r_0>0$ and $C>0$ such that, for all $x\in \partial K$ and $r\le r_0$,
\begin{equation*} 
\big|\vol_{\gamma}(B_{\gamma}(x,r))-\kappa_{d-1} r^{d-1}\big|
\le Cr^d.
\end{equation*}
\end{lem}

\begin{proof}
	Let $\varrho\in (0,\frac{\varrho_0}{2})$ and $B(p_1,\varrho_0),\dots,B(p_m,\varrho_0)$ be as in (the proof of) Lemma~\ref{lem:hd-covering}. Let $x\in \partial K$ and $r\le \frac{\varrho}{2}$. Then $B_{\gamma}(x,2r)$ belongs to a region $U=U^{(p_{\ell})}$ with parametrization $f=f^{(p_{\ell})}$ and $\proj(U)=\proj_{p_{\ell}}(U)= B(0,\lambda)$. Let $y\in B_{\gamma}(x,2r)$ and set $x'=\proj(x)$ and $y'=\proj(y)$. Define 
\begin{equation} \label{eq:geo-bar-0}
\bar{\gamma}(x,y)
=\inf_{u\in \curv} \bar{L}(u),
\qquad
\bar{L}(u):=\int_0^1 \sqrt{\langle G(x',x')\cdot u'(\tau),u'(\tau)\rangle} \dd \tau,
\end{equation}
where $\curv$ and $G$ are as in the proof of Lemma~\ref{lem:hd-covering} (and $\bar{\gamma}$ is as if $\xi=0$). It follows from  \eqref{eq:gamma-bar} that
\[
\bar{\gamma}(x,y)
=\sqrt{\langle G(x',x')\cdot (y'-x'),y'-x'\rangle}
\]
and similar to \eqref{eq:gammas} we have
\begin{equation} \label{eq:ell-app}
	|\gamma(x,y)-\bar{\gamma}(x,y)|\le \bar{C} r^2,
\end{equation}
where $\bar{C}>0$ depends only on $K$. Let $r_0<\min\{1,\sqrt{C_G}\}\frac{\varrho}{2}$ be small enough such that $r_0+\bar{C}r_0^2\le 2r_0$.  For $s\le 2r_0$, define the ellipsoid (recall that $G(x',x')$ is positive definite)
\begin{align*}
	\mathcal{E}(x',s)
&=\proj(\{y\in U\colon \bar{\gamma}(x,y)\le s\})\\
&=\{y'\in B(0,\lambda)\colon \langle G(x',x')\cdot (y'-x'),y'-x'\rangle \le s^2\},
\end{align*}
which is contained in $B(x',s/\sqrt{c_G})\subset B(0,2\lambda)$. Its volume is
\[
\vol(\mathcal{E}(x',s))=\frac{\kappa_{d-1}s^{d-1}}{\sqrt{\det q_{x'}}}.
\]
Assume that $r<r_0$. By \eqref{eq:ell-app} we have the inclusions 
\begin{equation} \label{eq:proj-ellipsoid}
	\mathcal{E}(x',r-\bar{C}r^2) \subset \proj(B_{\gamma}(x,r)) \subset \mathcal{E}(x',r+\bar{C}r^2).
\end{equation}
Similar as in the proof of Lemma~\ref{lem:hd-covering}, by Taylor's theorem and the bounds on the derivatives of $f$ it holds that, for any $u\in B(x',2r/\sqrt{c_G})$,
\[
\big|\det q_{u}-\det q_{x'}\big|
\le \sum_{\sigma}\Big|\prod_{i=1}^{d-1}G_{i,\sigma(i)}(u,u)-\prod_{i=1}^{d-1}G_{i,\sigma(i)}(x',x')\Big|
\le C r
\]
where the sum is over all permutations of $\{1,\dots,d-1\}$ and $C>0$ only depends on $K$ (and may change from line to line in the remainder of the proof). Therefore, by \eqref{eq:proj-ellipsoid} and Lemma~\ref{lem:pd-uniform}, it holds that
\begin{align*}
\vol_{\gamma}(B_{\gamma}(x,r))	
&=\int_{\proj(B_{\gamma}(x,r))}\sqrt{\det q_u}\dd u\\
&\le \int_{\mathcal{E}(x',r+\bar{C}r^2)}\sqrt{\det q_{u}}\dd u\\
&\le (1+C r)\sqrt{\det q_{x'}}\vol(\mathcal{E}(x',r+\bar{C}r^2))\\
&\le (1+C r)(1+\bar{C}r)^{d-1}\kappa_{d-1}r^{d-1},
\end{align*}
which gives 
\[
\vol_{\gamma}(B_{\gamma}(x,r))-\kappa_{d-1}r^{d-1}\le C r^d.
\]
The proof of the lower bound is analogous.
\end{proof}

For the proof of Theorem~\ref{thm:main}, we use the following limit theorem from \cite{Jan87}, see \cite{Jan86,Pen23} for related results. 

\begin{lem}\label{lem:janson}
Let $K\in \cK^3_+$. For each $n\in \IN$, let $X_1,\dots,X_n$ be independently distributed on $\partial K$ with density $h_{\kappa}$. Let $P_n=\{X_1,\dots,X_n\}$ and set 
\[
V(P_n)
=\sup\Big\{\frac{\vol_{\gamma}(B_{\gamma}(x,r))}{\vol_{\gamma}(\partial K)}\colon x\in \partial K, r>0 \text{ with } B_{\gamma}(x,r)\cap P_n=\emptyset\Big\}.
\]
Then, as $n\to \infty$, 
\begin{equation}
	\label{eq:V-gumbel}
nV(P_n)-\log n -(d-2)\log\log n-\log \alpha 
\xrightarrow[]{\rm d}
G,
\end{equation}
where $\alpha$ is as in Theorem~\ref{thm:main} and $G$ is a standard Gumbel distributed random variable.
\end{lem}

More precisely, Lemma~\ref{lem:janson} follows from the theorem and the second remark on p.~276  in \cite{Jan87} which refers to the argument employed in \cite[Section~8]{Jan86}.  Note that the results in \cite{Jan86,Jan87} were also used in \cite{Ser23} to obtain a limit theorem for the covering radius on the sphere and in \cite{GS96} to obtain \eqref{eq:gs-smooth}. 

\begin{proof}[Proof of Theorem~\ref{thm:main}]
	Let $P_n=\{X_1,\dots,X_n\}$ be as in the statement of the theorem. By Lemma~\ref{lem:vol-geo} the assumptions of \cite[Theorem~2.1]{RS16} are fulfilled and the geodesic covering radius satisfies
\begin{equation} \label{eq:covering-bound}
	\varrho(P_n)=O_{\IP}\Big(\Big(\frac{\log n}{n}\Big)^{\frac{1}{d-1}}\Big).
\end{equation}
Here and in the following, we write $X_n=O_{\IP}(a_n)$ for random variables $X_n$ and scalars $a_n$ if for every $\varepsilon>0$ there exists a constant $C>0$ such that $\IP[X_n\ge Ca_n]<\varepsilon$ for all $n$ large enough. It follows from Theorem~\ref{thm:hd-covering} and \eqref{eq:covering-bound} that
\begin{equation} \label{eq:delta-asymp}
\delta_H(K_n,K)
=\frac{\varrho(P_n)^2}{2}+O_{\IP}\Big(\Big(\frac{\log n}{n}\Big)^{\frac{3}{d-1}}\Big).
\end{equation}

In the following, we prove a limit theorem for $\varrho(P_n)^2$. By definition, every geodesic ball with radius larger than $\varrho(P_n)$ contains one of the points of $P_n$ and by compactness there is $x_0\in \partial K$ such that the ball $B_{\gamma}(x_0,\varrho(P_n))$ does not contain any point of $P_n$ in its interior. Thus, the random variable $V(P_n)$ defined in Lemma~\ref{lem:janson} is the Riemannian volume of the ball $B_{\gamma}(x_0,\varrho(P_n))$ divided by $\vol_{\gamma}(\partial K)=v_{\kappa}(K)$.  It follows from Lemma~\ref{lem:vol-geo} and \eqref{eq:covering-bound} that
\begin{equation*} 
V(P_n)
=\frac{\kappa_{d-1}}{v_{\kappa}(K)}\varrho(P_n)^{d-1}+ O_{\IP}\Big(\Big(\frac{\log n}{n}\Big)^{\frac{d}{d-1}}\Big).
\end{equation*}
Plugging this into the limit theorem \eqref{eq:V-gumbel} yields 
\begin{equation} \label{eq:rho-d-1}
\frac{\kappa_{d-1}}{v_{\kappa}(K)} n\varrho(P_n)^{d-1}-\log n -(d-2)\log\log n-\log \alpha 
\xrightarrow[]{\rm d}
G.
\end{equation}

Applying the uniform Delta Method \cite[Theorem~3.8]{VdV00} to the limit theorem in \eqref{eq:rho-d-1} and the function $x\mapsto \frac{d-1}{2}x^{\frac{2}{d-1}}$ gives
\begin{equation} \label{eq:limit-theorem}
	(d-1)\log n\Big[ \Big(\frac{\kappa_{d-1}}{v_{\kappa}(K)}\frac{n}{\log n}\Big)^{\frac{2}{d-1}} \frac{\varrho(P_n)^2}{2}-\frac{\theta_n}{2}\Big]\xrightarrow[]{\rm d} G,
\end{equation}
where 
\begin{align*}
\theta_n
&=\Big(1+\frac{(d-2)\log \log n+\log \alpha}{\log n}\Big)^{\frac{2}{d-1}}\\
&=1+\frac{2}{d-1}\frac{(d-2)\log \log n+\log \alpha}{\log n}+O\Big(\Big(\frac{\log \log n}{\log n}\Big)^2\Big).
\end{align*}
Combining the limit theorem in \eqref{eq:limit-theorem} with the asymptotics in \eqref{eq:delta-asymp} completes the proof of Theorem~\ref{thm:main}. 
\end{proof}

\subsection*{Acknowledgement}

This research was funded in whole or in part by the Austrian Science Fund (FWF) [Grant DOI: 10.55776/J4777]. For open access purposes, the author has applied a CC BY public copyright license to any author-accepted manuscript version arising from this submission. 

\bibliographystyle{abbrv}
\bibliography{nclt_convex}

\end{document}